%% file: Omega1G.tex
\documentclass{amsart}
\usepackage{todonotes}
\usepackage{xypic}
\usepackage{graphicx, amsmath, amssymb, amsthm}
\usepackage{hyperref} 

\input{MAHMacros.tex}

\title[Equivariant Derivations]{On the Andr\'e-Quillen homology of Tambara functors}

\author{Michael A.~Hill}
\thanks{\support{\NSFThree}}

\begin{document}

\begin{abstract}
We lift to equivariant algebra three closely related classical algebraic concepts: abelian group objects in augmented commutative algebras, derivations, and K\"ahler differentials. We define Mackey functor objects in the category of Tambara functors augmented to a fixed Tambara functor $\underline{R}$, and we show that the usual square-zero extension gives an equivalence of categories between these Mackey functor objects and ordinary modules over $\underline{R}$. We then describe the natural generalization to Tambara functors of a derivation, building on the intuition that a Tambara functor has products twisted by arbitrary finite $G$-sets, and we connect this to square-zero extensions in the expected way. Finally, we show that there is an appropriate form of K\"ahler differentials which satisfy the classical relation that derivations out of $\underline{R}$ are the same as maps out of the K\"ahler differentials.
\end{abstract}

\maketitle

\section{Introduction}
Foundational work of Andr\'e and Quillen defined notions of homology and cohomology for commutative rings \cite{Andre67}, \cite{Quillen70}. This provided a natural way to understand the deformations of a commutative ring, connecting them to derivations, providing a condition for \'etale-ness, and building a natural long-exact sequence analogous to those from topology for a triple. Unpublished work of Kriz lifted this to structured ring spectra, showing that certain Postnikov invariants can be recast as Andr\'e-Quillen cohomology groups \cite{KrizTAQ}. Basterra extended this, producing the theory of topological Andr\'e-Quillen homology of a commutative ring spectrum \cite{Basterra99}. This work was then extended by Basterra-Mandell, who showed that $\TAQ$ with coefficients is essentially the only homology theory on commutative ring spectra and who explored the basics of spectrum objects in commutative ring spectra \cite{BasterraMandell}.

In the $G$-equivariant context for a finite group $G$, the role of abelian groups in non-equivariant algebra is played by Mackey functors. The category of Mackey functors is a closed symmetric monoidal category with symmetric monoidal product, the box product. In addition to the expected generalization of commutative rings to simply commutative monoids for the box product, there is a poset of generalizations of the notion of commutative rings to the $G$-equivariant context: the incomplete Tambara functors \cite{BHIncomplete}. These interpolate between Green functors, the ordinary commutative monoids for the box product, and Tambara functors \cite{Tambara}. The distinguishing feature for [incomplete] Tambara functors is the presence of certain multiplicative transfer maps, called norm maps. For a Green functor, we have no norm maps; for a Tambara functor, we have norm maps for any pair of subgroups $H\subset K$ of $G$.

This paper explores three closely related themes from classical commutative algebra in the setting of Tambara functors: square-zero extensions, derivations, and K\"ahler differentials.  Strickland initiated this study, showing that in stark contrast to the classical case, Quillen's abelian group objects in Tambara functors over a fixed Tambara functor $\m{R}$ properly contains the category of $\m{R}$-modules. In particular, the Andr\'e-Quillen homology groups are in general more complicated than simply the derived functors of derivations into an $\m{R}$-module. In this paper, we explain how to rectify this situation, showing that the correct analogue of the abelian group objects is the Mackey functor objects:

\begin{theorem*}
The square-zero extension gives an equivalence of categories between the category of $\m{R}$-modules and the category of Mackey functor objects in the category of $\m{S}$-Tambara functors augmented to $\m{R}$.
\end{theorem*}

Classically, maps into a square-zero extension are classified by derivations, and with the appropriate notion, such a thing is true here. Classically, a derivation turns sums to products. We define below (Definition~\ref{def:GenuineDerivation}) a ``genuine derivation'' which plays the equivariant role, converting  twisted products (the norms) into twisted sums (the transfers).

\begin{theorem*}
The set of maps from an $\m{S}$-Tambara functor $\m{C}$ augmented to $\m{R}$ to a square-zero extension $\m{R}\semidirect\mM$ is naturally isomorphic to the set of genuine $\m{S}$-derivations of $\m{C}$ into $\mM$.
\end{theorem*}

Finally, there is an $\m{R}$-module of genuine K\"ahler differentials (Definition~\ref{def:GenuineKahler}) which receives the universal genuin $S$-derivation from $\m{R}$.

\begin{theorem*}
There is an $\m{R}$-module $\OmegaOneG$ and a universal genuine $\m{S}$-derivation $d\colon\m{R}\to\OmegaOneG$. This has the property that genuine $S$-derivations from $\m{R}$ to an $\m{R}$-module $\mM$ are in natural bijective correspondence with $\m{S}$-module maps $\OmegaOneG\to\mM$.
\end{theorem*}

\subsection*{Notational Conventions}
In this paper, $G$ will always denote a finite group. We will usually reserve the letters $H$ and $K$ for subgroups of $G$. Additionally, we will denote coefficient systems, Mackey functors, Tambara functors, and related constructions with underlined capital Roman letters to distinguish them from the non-equivariant objects.

\subsection*{Acknowledgements}
We thank Andrew Blumberg and Tyler Lawson for several helpful conversations and for their careful reading of earlier drafts.

\section{Brief review of Tambara functors}
\subsection{Ordinary Tambara functors}

\begin{definition}
Let $\cP^G$ denote the category of polynomials in $G$-sets. The objects are finite $G$-sets, and the morphisms are isomorphism classes of diagrams
\[
S\xleftarrow{f} U\xrightarrow{g} V\xrightarrow{h} T,
\]
where two such diagrams are isomorphic if we have a commutative diagram
\[
\xymatrix@R=.1\baselineskip{
{} & {U'}\ar[r]^{g'}\ar[dd]_{\cong}\ar[ld]_{f'} & {V'}\ar[dd]^{\cong}\ar[rd]^{h'} & {} \\
{S} & {} & {} & {T.} \\
{} & {U}\ar[r]_{g}\ar[lu]^{f} & {V}\ar[ru]_{h} & {}
}
\]
\end{definition}
Composition in this category is a bit trickier to describe, so it is convenient to name a generating collection of morphisms and then describe their commutation relations.
\begin{definition}
Let $f\colon S\to T$ be a map of finite $G$-sets. Then let
\begin{align*}
R_f&:=[T\xleftarrow{f} S\xrightarrow{=} S\xrightarrow{=}S] \\
N_f&:=[S\xleftarrow{=} S\xrightarrow{f} T\xrightarrow{=}T] \\
T_f&:=[S\xleftarrow{=} S\xrightarrow{=} S\xrightarrow{f} T]
\end{align*}
\end{definition}

Then any polynomial can be written as a composite of these:
\[
T_h\circ N_g\circ R_f=[S\xleftarrow{f} U\xrightarrow{g} V\xrightarrow{h} T].
\]
These have the following relations.
\begin{proposition}
$R$ gives a contravariant functor from $\Set^G$ into $\cP^G$. $N$ and $T$ give covariant ones.
\end{proposition}

\begin{proposition}
If we have a pullback diagram of finite $G$-sets
\[
\xymatrix{
{S'}\ar[r]^{f'}\ar[d]_{g'} & {T'}\ar[d]^{g} \\
{S}\ar[r]_{f} & {T,}
}
\]
then we have
\[
R_g\circ N_f=N_{f'}\circ R_{g'}\text{ and }R_g\circ T_f=T_{f'}\circ R_{g'}.
\]
\end{proposition}

The interchange of $N$ and $T$ is trickier. Recall that if $f\colon S\to T$ is a map of finite $G$-sets, then the pullback functor
\[
f^\ast\colon\Set^G_{\downarrow T}\to\Set^G_{\downarrow S}
\]
has a right adjoint: the dependent product $\prod_{f}$. 
\begin{definition}
An exponential diagram in $\Set^G$ is a diagram (isomorphic to one) of the form
\[
\xymatrix{
{S}\ar[d]_{h} & {A}\ar[l]_g & {S\times_{T}\prod_h A}\ar[l]_-{f'}\ar[d]^{g'} \\
{T} & & {\prod_h A.}\ar[ll]^{h'}
}
\]
\end{definition}

\begin{proposition}
If we have an exponential diagram
\[
\xymatrix{
{S}\ar[d]_{g} & {A}\ar[l]_h & {S\times_{T}\prod_g A}\ar[l]_-{f'}\ar[d]^{g'} \\
{T} & & {\prod_g A,}\ar[ll]^{h'}
}
\]
then 
\[
N_g\circ T_h=T_{h'}\circ N_{g'}\circ R_{f'}.
\]
\end{proposition}

With these morphisms, the disjoint union of finite $G$-sets becomes the product in the category $\cP^G$.

\begin{definition}
A semi-Tambara functor is a product preserving functor $\cP^G\to\Set$. A Tambara functor is a semi-Tambara functor $\m{R}$ for which $\m{R}(T)$ is group-complete for all $T\in\Set^G$.
\end{definition}
Tambara showed that the group-completion functor can be applied to any semi-Tambara functor, giving a Tambara functor.

There are several related categories of polynomials which give other flavors of Tambara functors. Recall that a subgraph of a category $\cC$ is ``wide'' if it contains all of the objects.
\begin{definition}
Inside the category $\cP^G$ are three important wide sub-graphs:
\begin{enumerate}
\item $\cP^{G}_{Iso}$ where the map $g$ in a polynomial is an isomorphism,
\item $\cP^{G}_{Epi}$ where the map $g$ in a polynomial is an epimorphism, and
\item $\cP^{G}_{gr}$ where the map $g$ in a polynomial preserves isotropy in the sense that for all $u\in U$, the stabilizer of $g(u)$ is that of $u$.
\end{enumerate}
\end{definition}

\begin{proposition}[{\cite[Prop. 2.12]{BHIncomplete}}]
The subgraphs $\cP^{G}_{Iso}$, $\cP^{G}_{Epi}$, and $\cP^{G}_{gr}$ are subcategories of $\cP^{G}$ in which the disjoint union of finite $G$-sets is the product.
\end{proposition}

\begin{proposition}[{\cite[Prop. 4.3]{BHIncomplete}}]
A product preserving functor $\cP^{G}_{Iso}\to\Set$ is a semi-Mackey functor. 
\end{proposition}

\begin{proposition}[{\cite[Prop. 12.11]{Strickland}}]
A product preserving functor $\cP^{G}_{gr}\to\Set$ is a semi-Green functor.
\end{proposition}

The category of Mackey functors is a closed symmetric monoidal category. The symmetric monoidal product is called the box product and is the Day convolution product of the tensor product of abelian groups with the Cartesian product of finite $G$-sets. Classically, a commutative Green functor is a commutative monoid under the box product. In particular, there is an obvious notion of the category of modules over a Green functor, and this is a symmetric monoidal category if the Green functor is commutative.

Expanding out what it means to be a commutative monoid under the box product, we see that a [commutative] Green functor is a Mackey functor $\m{R}$ such that for all finite $G$-sets $T$, $\m{R}(T)$ is commutative ring, such that all restriction maps are maps of commutative rings, and such that if $f\colon T\to T'$ is a map of finite $G$-sets, then we have the Frobenius reciprocity relation
\[
a\cdot T_f(b)=T_f(R_f(a)\cdot b)
\]
for all $a\in \m{R}(T')$ and $b\in\m{R}(T)$. 

There is a similar description for Tambara functors. 

\begin{proposition}[{\cite{MazurArxiv}}]
A Tambara functor is a commutative Green functor $\m{R}$ together with norm maps
\[
N_H^K\colon \m{R}(G/H)\to\m{R}(G/K)
\]
for all $H\subset K\subset G$.
These are maps of multiplicative monoids and they satisfy certain universal formulae expressing the norm of a sum and the norm of a transfer.
\end{proposition}
The exact formulae for the norms of a transfer will not matter for us here; it suffices that such a formula exists. For a sum, we need slightly more information.

\begin{proposition}[{\cite[Thm. 2.3]{MazurArxiv}}]\label{prop:NormofSum}
Consider the maps $\nabla\colon G/H\amalg G/H\to G/H$ and $\pi\colon G/H\to\ast$. Then we have an isomorphism of $G$-sets over $\ast$
\[
\prod_{\pi} \nabla\cong \Big(F(G/H,\{0,1\})\to\ast\Big),
\]
where $\{0,1\}=(G/H\amalg G/H)/G$ has a trivial action. 

The diagram
\[
\xymatrix{
{G/H}\ar[d]_{\pi} & {G/H\amalg G/H}\ar[l]_-{\nabla} & {G/H\times F(G/H,\{0,1\})}\ar[l]_-{\epsilon}\ar[d]^{g} \\
{\ast} & & {F(G/H,\{0,1\})}\ar[ll]^{h}
}
\]
is an exponential diagram, where 
\[
\epsilon(gH,f):=(gH,f(gH))\in G/H\times\{0,1\}\cong G/H\amalg G/H.
\]
\end{proposition}

Proposition~\ref{prop:NormofSum} gives the formula for the norm of a sum of elements:
\[
N_H^G(a+b)=T_h\circ N_g\circ R_f(a,b).
\]

When discussing differentials and the universal differential, we will need to work with non-unital Tambara functors. These can be defined simply from $\cP^{G}_{Epi}$.

\begin{definition}
A {\defemph{non-unital semi-Tambara functor}} is a product preserving functor $\cP^{G}_{Epi}\to\Set$. It is a {\defemph{non-unital Tambara functor}} if it is group complete.
\end{definition}

Just as with ordinary Tambara functors, we can view a non-unital Tambara functor as a non-unital Green functor together with norm maps that satisfy the same universal formulae.

\subsection{Relative Tambara functors}
If $\m{S}$ is a Tambara functor, then we can talk about Tambara functors and non-unital Tambara functors in the category of $\m{S}$-modules.

\begin{definition}
If $\m{S}$ is a Tambara functor, then an {\defemph{$\m{S}$-Tambara functor}} is a Tambara functor $\m{R}$ together with a map $\m{S}\to\m{R}$ of Tambara functors.

Let $\STamb$ denote the corresponding comma category of Tambara functors equipped with a map from $\m{S}$.
\end{definition}

\begin{definition}\label{def:NonUnitalRAlgebra}
A {\defemph{non-unital $\m{S}$-Tambara functor}} is an $\m{S}$-module $\m{R}$ equipped with norm maps for any surjection $f\colon T\to T'$ that satisfies
\[
N_f(r\cdot s)=N_f(r)\cdot N_f(s)
\]
for all $s\in\m{S}(T)$ and $r\in\m{R}(T)$.
\end{definition}

Both of these have a more diagrammatic approach.

\begin{proposition}
Let $\m{S}$ be a Tambara functor and let $\m{R}$ be a [non-unital] Tambara functor. Assume that $\m{R}$ is a module over $\m{S}$, and let 
\[
\mu\colon\m{S}\Box\m{R}\to\m{R}
\] 
be the action of $\m{S}$ on $\m{R}$. Then $\m{R}$ is a [non-unital] $\m{S}$-Tambara functor if and only if $\mu$ is a map of [non-unital] Tambara functors.
\end{proposition}

\begin{remark}
The category of modules over a Tambara functor $\m{S}$ inherits a $G$-symmetric monoidal structure from the category of Mackey functors. The $G$-commutative monoids here are exactly the $\m{S}$-Tambara functors, and the non-unitial $G$-commutative monoids are exactly the non-unital $\m{S}$-Tambara functors.
\end{remark}

\section{Abelian group and Mackey functor objects}

We recall work of Strickland (building on work of Quillen) on the homology of a Tambara functor.

\begin{definition}
Let $\m{R}$ be an $\m{S}$-Tambara functor.

Let $\STamb_{/\m{R}}$ be the comma category of $\m{S}$-Tambara functors with a map to $\m{R}$. 

Let $\SAb_{/\m{R}}$ denote the category of abelian group objects in $\STamb_{/\m{R}}$. 

Let $\Mods{\m{R}}$ denote the category of modules over the underlying Green functor for $\m{R}$ in the category of Mackey functors. 
\end{definition}

There is an obvious ``augmentation ideal'' functor
\[
I\colon\SAb_{/\m{R}}\to\Mods{\m{R}}
\]
which assigns to an abelian group object $\m{B}$ the kernel of $\m{B}\to\m{R}$. In commutative rings, this functor is half of an equivalence of categories, with quasi-inverse given by the square-zero extension. Strickland shows that square-zero extensions make perfect sense here, but that these are not inverse equivalences.

\begin{proposition}[{\cite[Prop. 14.7]{Strickland}}]
There is a ``square-zero extension functor''
\[
\m{R}\semidirect(-)\colon \Mods{\m{R}}\to \SAb_{/\m{R}}
\]
which sends an $\m{R}$-module to the square-zero extension in Green functors and which endows the module summand with trivial norms.

These are not inverse equivalences: the map $\m{R}\semidirect(-)$ is not essentially surjective.
\end{proposition}

In the square-zero extension, the $\m{S}$-Tambara functor structure is induced by the natural maps of Tambara functors
\[
\m{S}\xrightarrow{\eta}\m{R}\xrightarrow{Id\semidirect 0}\m{R}\semidirect\mM.
\]

The issue here is with norms in the augmentation ideal. The only condition we deduce from this being an abelian group object is that all products vanish. However, this only tells us about the restrictions of norms to various subgroups, not to the norms themselves. To better explain the failure of this equivalence and to prove the more accurate statement, we being with a simple observation.

\begin{proposition}
If $\m{R}$ and $\m{B}$ are Tambara functors, then the set of Tambara functor maps between them has a natural extension to a coefficient system of sets:
\[
\m{\Tamb}(\m{R},\m{B})(G/H)=\Tamb^H(i_H^\ast\m{R},i_H^\ast\m{B})\subset\Mackey^H(i_H^\ast\m{R},i_H^\ast\m{B}).
\]
The restriction maps on Mackey functors give rise to the restriction maps in $\m{\Tamb}$.  This provides an enrichment in coefficient systems for the category $\Tamb$, where composition and the units are level-wise.
%{\bf AJB: Say something about composition and unit maps?}

The categories $\STamb$ and $\STamb_{/\m{R}}$ are also enriched in coefficient systems and form a sub-coefficient system of $\Tamb$.
\end{proposition}
The following is an immediate application of the Yoneda Lemma.

\begin{proposition}
An abelian group structure on $\m{B}\to\m{R}$ is the same as a natural lift of $\m{S}\mhyphen\m{\Tamb}_{/\m{R}}(-,\m{B})$ to a coefficient system of abelian groups.
\end{proposition}

The Yoneda Lemma also better explains the coefficient system structure here. The restriction functor $i_H^\ast$ from $G$-Tambara functors augmented over $\m{R}$ to $H$-Tambara functors augmented over $i_H^\ast{\m{R}}$ has a right adjoint: coinduction \cite[Prop. 18.3]{Strickland}. This has a very simple formulation: for any $T\in\Set^G$, 
\[
\CoInd_H^G(\m{R})(T):=\m{R}(i_H^\ast T).
\] 
Similarly, if $f\colon T\to T'$, then 
\begin{align*}
T_f&:= T_{i_H^\ast f}\\
N_f&:= N_{i_H^\ast f} \\
R_f&:= R_{i_H^\ast f}.
\end{align*}

Since $\CoInd_H^G$ is the right adjoint to $i_H^\ast$, we have a natural map of Tambara functors
\[
\eta_{\m{R}}\colon \m{R}\to\CoInd_H^G i_H^\ast\m{R}. 
\]

This gives us the right adjoint to $i_H^\ast$ in the category $\STamb$: if $\m{R}$ is an $i_H^\ast\m{S}$-Tambara functor, then $\CoInd_H^G\m{R}$ is an $\m{S}$-Tambara functor via the composite
\[
\m{S}\xrightarrow{\eta_{\m{S}}} \CoInd_H^Gi_H^\ast\m{S}\xrightarrow{\CoInd_H^G\eta}\CoInd_H^G\m{R}.
\]

We can also define a relative version of coinduction.
\begin{definition}
If $\m{B}\xrightarrow{f} i_H^\ast\m{R}$ is a Tambara functor over $i_H^\ast\m{R}$, then let $F_H(G,\m{B})$ be the pullback
\[
\xymatrix{
{F_H(G,\m{B})}\ar[r]\ar[d]_{F_H(G,f)} & {\CoInd_H^G\m{B}}\ar[d]^{\CoInd_H^G f}  \\
{\m{R}}\ar[r]_-{\eta_{\m{R}}} & {\CoInd_H^Gi_H^\ast\m{R}.}
}
\]
\end{definition}

\begin{proposition}
If $\m{B}$ is an $i_H^\ast\m{S}$-Tambara functor and $\m{R}$ is an $\m{S}$-Tambara functor, then the pullback of the structure maps gives $F_H(G,\m{B})$ the structure of an $\m{S}$-Tambara functor.
\end{proposition}
\begin{proof}
Consider the diagram
\[
\xymatrix{
{\m{S}}\ar[r]^-{\eta_{\m{S}}}\ar[d]_{\eta_{\m{R}}} & {\CoInd_H^Gi_H^\ast\m{S}}\ar[r]^{\CoInd_H^G \eta_{\m{B}}} \ar[d]_{\CoInd_H^G \eta_{i_H^\ast \m{R}}} & {\CoInd_H^G\m{B}}\ar[dl]^{\CoInd_H^G\epsilon} \\
{\m{R}}\ar[r]_-{\eta_{\m{R}}} & {\CoInd_H^G i_H^\ast\m{R}.}
}
\]
The square commutes since $\eta$ is a natural tranformation. The triangle commutes since $\m{B}$ is an $i_H^\ast\m{S}$-Tambara functor augmented to $i_H^\ast\m{R}$.
\end{proof}

\begin{proposition}
The functor $F_H(G,-)$ is the right-adjoint to the restriction functor $i_H^\ast$ in the category of Tambara functors augmented over $\m{R}$.
\end{proposition}

The unit of the restriction-coinduction adjunction is induced by the natural commutative square
\[
\xymatrix{
{\m{B}}\ar[d]_f \ar[r]^{\eta_{\m{B}}} & {\CoInd_H^G i_H^\ast\m{B}}\ar[d]^{\CoInd_H^G i_H^\ast f} \\
{\m{R}}\ar[r]_{\eta_{\m{R}}} & {\CoInd_H^G i_H^\ast\m{R}.}
}
\]

The Yoneda Lemma now also describes the restriction maps in the coefficient system $\m{S}\mhyphen\m{\Tamb}$.

\begin{proposition}
The restriction maps in
\[
\m{S}\mhyphen\m{\Tamb}_{/\m{R}}(\m{C},\m{B})
\]
are induced by the natural maps $\eta_{\m{B}}\colon\m{B}\to F(G/H,\m{B})$.
\end{proposition}

To fully understand the structure, we extend this coefficient system in the obvious way to a product preserving functor
\[
\m{S}\mhyphen\m{\Tamb}_{/\m{R}}(\m{C},\m{B})\colon\big(\Set^{G,\amalg}\big)^{op}\to\Set.
\]
This part is also representable. 

\begin{proposition}[{\cite[Cor. 6.7]{BHIncomplete}}]
If $\m{B}$ is a Tambara functor and $T$ is a finite $G$-set, then the Mackey functor
\[
\m{B}_T:=\m{B}(T\times -)
\]
has a canonical Tambara functor structure. 

When $T=G/H$, we have a natural isomorphism
\[
\m{B}_{G/H}\cong \CoInd_H^G i_H^\ast\m{B}
\]
\end{proposition}

Since the Cartesian product distributes over disjoint union, the following is immediate.
\begin{proposition}
If $\m{B}$ is a Tambara functor and $T_1$ and $T_2$ are finite $G$-sets, then we have a natural isomorphism of Tambara functors
\[
\m{B}_{T_1\amalg T_2}\cong \m{B}_{T_1}\times\m{B}_{T_2}.
\]
\end{proposition}

Combining this with the units of the restriction-coinduction adjunction then gives the following.
\begin{proposition}
If $\m{B}$ is a Tambara functor, then for any finite $G$-set $T$, there is a natural map of Tambara functors
\[
\m{B}\to\m{B}_T.
\]
In particular, if $\m{B}$ is an $\m{S}$-Tambara functor, then $\m{B}_T$ is canonically so for any $T$.
\end{proposition}

Using all of this we can define a version of this in the category of $\m{S}$-Tambara functors augmented to $\m{R}$.

\begin{definition}
If $\m{B}\to\m{R}$ is an $\m{S}$-Tambara functor augmented to $\m{R}$ and if $T$ is a finite $G$-set, then let $F(T,\m{B})$ be the pullback
\[
\xymatrix{
{F(T,\m{B})}\ar[r]\ar[d] &  {\m{B}_{T}} \ar[d] \\
{\m{R}}\ar[r] & {\m{R}_{T}.}
}
\]
\end{definition}

\begin{proposition}\label{prop:Semiring}
If $\m{B}$ is an $\m{S}$-Tambara functor augmented to $\m{R}$ and if $T_1$ and $T_2$ are finite $G$-sets, then we have a natural isomorphism
\[
F(T_1\times T_2,\m{B})\cong F\big(T_2,F(T_1,\m{B})\big).
\]
\end{proposition}
\begin{proof}
Since the Cartesian product of finite $G$-sets is associative up to natural isomorphism, we have a natural isomorphism
\[
(\m{B}_{T_1})_{T_2}\cong \m{B}_{T_1\times T_2}.
\]
The result then follows from observing that both Tambara functors are the pullback of the diagram 
\[
\xymatrix{
{} & {\m{B}_{T_1\times T_2}}\ar[d] \\
{\m{R}}\ar[r] & {\m{R}_{T_1\times T_2}.} 
}
\] 
\end{proof}

Having symmetric monoidal functors which act as symmetric monoidal powers indexed by a $G$-set is exactly one of the ways to parse the notion of a $G$-symmetric monoidal category \cite[Def. 3.3]{HHLocalization}, so we conclude the following \cite{HHLocalization}.

\begin{theorem}
With coinduction as categorical transfer maps, the category of Tambara functors augmented over $\m{R}$ becomes a $G$-symmetric monoidal category.  The internal tensoring with a finite $G$-set $T$ is given by the functors $F(T,-)$.
%Moreover, the forgetful functor 
%\[
%\m{\Tamb}_{/\m{R}}\to\m{\Mackey}_{\downarrow \m{R}}
%\]
%from Tambara functors augmented over $\m{R}$ to the comma category of Mackey functors over the underlying Mackey functor for $\m{R}$ is a strong $G$-symmetric monoidal functor.
\end{theorem}

This lets us reformulate Strickland's definition. In some sense, this proposition has no real content: it is an immediate reformulation of Strickland's result.

\begin{proposition}
The category $\SAb_{/\m{R}}$ is the category of group-like commutative monoids in $\m{S}\mhyphen\m{\Tamb}_{/\m{R}}$.
\end{proposition}

Since $\m{\Tamb}_{/\m{R}}$ is a $G$-symmetric monoidal category, we have a notion of $G$-commutative monoids \cite[Def. 3.8]{HHLocalization}. 

%\begin{definition}
%The $G$-Cartesian monoidal structure on the category of set-valued coefficient systems is the coinduction $G$-symmetric monoidal category structure. The internal exponentiation by a finite $G$-set $T$ is given by
%\[
%\m{C}\mapsto \m{C}_T:=(T'\mapsto\m{C}(T\times T').
%\]
%\end{definition}
%

\begin{proposition}
If $\m{B}\to\m{R}$ is a group-like $G$-commutative monoid in $\m{\Tamb}_{/\m{R}}$, then for all $\m{C}\to\m{R}$, the coefficient system
\[
\m{\Tamb}_{/\m{R}}(\m{C},\m{B})
\]
has natural extension to a Mackey functor.
\end{proposition}
\begin{proof}
Let $\m{C}\to\m{R}$ be a Tambara functor augmented to $\m{R}$, and let 
\[
\m{B}_{\m{C}}:=\m{S}\mhyphen\m{\Tamb}_{/\m{R}}(\m{C},\m{B})
\]
be the coefficient system in question. %Since $\m{B}$ is a group-like, this is naturally a coefficient system of abelian groups. 
By construction, the value of this at a finite $G$-set $T$ is given by
\[
\m{B}_{\m{C}}(T):=\STamb_{/\m{R}}\big(\m{C},F(T,\m{B})\big).
\]
In particular, Proposition~\ref{prop:Semiring} shows that we have a natural isomorphism of coefficient systems
\[
F(T,\m{B})_{\m{C}}\cong N^T(\m{B}_{\mC}),
\]
where $N^T$ is the endo-functor on coefficient systems of sets given by
\[
(N^T\mM)(T'):=\mM(T\times T').
\]

By naturality, the $G$-commutative monoid structure of $\m{B}$ makes $\m{B}_{\m{C}}$ a $G$-commutative monoid in the coinduction $G$-symmetric monoidal structure on coefficient systems. By \cite[Thm. 5.6]{HHLocalization}, this is exactly a Mackey functor structure on $\m{B}_{\m{C}}$.
\end{proof}

\begin{definition}
A Mackey functor object in $\Tamb_{/\m{R}}$ is a group-like $G$-commutative monoid in $\m{\Tamb}_{/\m{R}}$. The category of Mackey functor objects and maps is denoted $\Mackey_{/\m{R}}$.
\end{definition}

We can immediately produce a collection of such objects. Recall that a strong $G$-symmetric monoidal functor between $G$-symmetric monoidal categories is one for which we have natural isomorphism
\[
F\big(N^T(-)\big)\Rightarrow N^T\big(F(-)\big).
\]

\begin{proposition}
The functor
\[
\m{R}\semidirect(-)\colon\Mods{\m{R}}\to\STamb_{/\m{R}}
\]
is a strong $G$-symmetric monoidal functor.
\end{proposition}
\begin{proof}
The underlying Mackey functors for $\CoInd_H^G$ and for $F(T,-)$ are determined by the corresponding functors on Mackey functors. In this case, we have natural isomorphisms of Mackey functors augmented to $\m{R}$:
\[
F\big(T,\m{R}\oplus \mM\big)\cong\m{R}\oplus \mM_T.
\]
In both cases, the augmentation ideal has trivial norms and products, meaning that this identification is also one of Tambara functors.
\end{proof}
\begin{corollary}
The functor 
\[
\m{R}\semidirect(-)\colon\Mods{\m{R}}\to\STamb_{/\m{R}}
\]
lifts to a functor to $\Mackey_{/\m{R}}$.
\end{corollary}
\begin{proof}
Any Mackey functor is a group-like $G$-commutative monoid. A strong $G$-symmetric monoidal functor preserves these. 
\end{proof}

We would like to better understand the category of Mackey functor objects augmented to $\m{R}$, and for this, we unpack some the externalized transfer maps. It is helpful to compare these with the transfer maps in the underlying Mackey functors.

\begin{lemma}\label{lem:MackeyUniqueness}
Any Mackey functor has a unique structure as a $G$-commutative monoid.
\end{lemma}
\begin{proof}
In Mackey functors, coinduction and induction agree. In particular, $CoInd_H^G$ is the left-adjoint to the forgetful functor as well as the right, and hence a map
\[
F(G/H,\mM)=\CoInd_H^G i_H^\ast\mM\xrightarrow{\mTr_H^G}\mM
\]
is determined by its adjoint $i_H^\ast\mM\to i_H^\ast \mM$. The adjoint can be computed as
\[
i_H^\ast\mM\to i_H^\ast CoInd_H^G i_H^\ast\mM\cong i_H^\ast F(G/H,\mM)\cong F(i_H^\ast G/H,i_H^\ast\mM)\xrightarrow{i_H^\ast \mTr_H^G} i_H^\ast\mM,
\]
where the first map is the unit of the adjunction. This corresponds to the inclusion $H/H\hookrightarrow i_H^\ast G/H$, and the composite is then just the identity map. Thus $\mTr_H^G$ must be the adjoint to the identity map on $i_H^\ast\mM$, and hence is uniquely determined.
\end{proof}
%\begin{remark}
%This lemma is the reason for calling this $G$-symmetric monoidal structure $G$-Cartesian: it is the obvious equivariant analogue of a Cartesian monoidal structure and the commutative monoids therein. Work of Berman has explored this extensively \cite{Berman}.
%\end{remark}

\begin{corollary}\label{cor:TransfersAreMorphisms}
If $\m{B}\in\Mackey_{/\m{R}}$, then all external transfer maps in $\m{B}$ are maps of Tambara functors.
\end{corollary}
\begin{proof}
This is an immediate consequence of Lemma~\ref{lem:MackeyUniqueness}.
\end{proof}

This reformulation allows us to be explain the discrepancy seen by Strickland for abelian group objects.

\begin{theorem}
If an $\m{S}$-Tambara functor $\m{B}\to\m{R}$ is a group-like $G$-commutative monoid in $\m{\Tamb}_{/\m{R}}$, then all norms and products in the non-unital Tambara functor $I(\m{B})$ are zero.
\end{theorem}
\begin{proof}
Since the underlying product is zero in ordinary group-like commutative monoids, the only possible norms that we would have are those of the form
\[
N_H^K=N_{G/H\to G/K}.
\]
This map is determined by the norm in $N_H^K=N_{K/H\to K/K}$ in $i_K^\ast\m{B}$, so it suffices to assume that $K=G$. We therefore have to show that for any $a\in \m{B}(G/H)$, $N_H^G(a)=0$.

Consider the map of Tambara functors
\[
F(G/H,\m{B})\xrightarrow{\mTr_H^G}\m{B}.
\]
By Corollary~\ref{cor:TransfersAreMorphisms}, this is the Mackey refinement of the ordinary transfer on $I(\m{B})$. In particular, at level $G/H$, the map is surjective. However, in $F(G/H,\m{B})$, the map $N_H^G$ is identically zero:
\[
N_H^G=N_{G/H\to\ast}:=N_{i_H^\ast G/H\to\ast}=N_{\ast\amalg T\to\ast}=\mu\circ Id\times N_{T\to\ast}=0,
\]
where $\mu$ is the multiplication, where $T=(i_H^\ast G-H)/H$, and where we have used that the underlying Green functor has trivial products.
\end{proof}

\begin{corollary}
The functors
\[
I\colon \Mackey_{/\m{R}} \rightleftarrows \m{R}\mhyphen\Mod\colon \m{R}\semidirect(-)
\]
are inverse equivalences of categories.
\end{corollary}

\section{Genuine Derivations}

\begin{definition}\label{def:GenuineDerivation}
Let $\m{S}$ and $\m{R}$ be Tambara functors, $\eta\colon\m{S}\to\m{R}$ a map of Tambara functors, and let $\mM$ be an $\m{R}$-module. We say that a map
\[
d\colon\m{R}\to\mM
\]
is a {\defemph{genuine $\m{S}$-derivation}} if 
\begin{enumerate}
\item for all finite $G$-sets $T$ and all $r_1, r_2\in \m{R}(T)$, we have 
\[
d(r_1\cdot r_2)=r_1\cdot d(r_2)+d(r_1)\cdot r_2\in\m{M}(T),
\]
\item for all $a\in\m{R}(G/H)$, 
\[
d(N_H^Ka)=tr_H^K N_{d_2}R_{d_1}(a)\cdot d(a),
\]
where $d_i$ is the restriction of the projection onto the $i$th factor of the complement of the diagonal in $K/H\times K/H$, and
\item $d\circ \eta=0$.
\end{enumerate}

Let $\Der_{\m{S}}^{1,G}(\m{R},\mM)$ be the set of all genuine $\m{S}$-derivations from $\m{R}$ to $\mM$.
\end{definition}

The intuition here is that just as an ordinary derivation turns ordinary multiplications into sums, a genuine derivation turns twisted multiplications (norms) into twisted sums (transfers).  

The following is immediate from the definitions.
\begin{proposition}\label{prop:Naturality}
Let $d\colon\m{R}\to\mM$ be a genuine derivation. 
\begin{enumerate}
\item If $\iota\colon \m{S}\to\m{R}$ is a map of Tambara functors, then $d\circ \iota$ is a genuine derivation, where $\mM$ is viewed as an $\m{S}$-module via $\iota$.
\item If $f\colon \mM\to\mM'$ is a map of $\m{R}$-modules, then $f\circ d$ is a genuine derivation.
\end{enumerate}
\end{proposition}

\begin{proposition}
If $\m{R}$ is an $\m{S}$-Tambara functor, $\m{M}$ is an $\m{R}$-module, and $d\colon\m{R}\to\mM$ is a genuine $\m{S}$-derivation, then $\ker(d)$ is a sub-Tambara functor of $\m{R}$.
\end{proposition}
\begin{proof}
Since $d$ is an ordinary derivation, $\ker(d)$ is a sub-Green functor of $\m{R}$. If $a\in\ker(d)(G/H)$, then since $d$ is a genuine $\m{S}$-derivation, 
\[
d\big(N_H^Ka\big)=\tr_H^K \big(N_{d_2} R_{d_1}(a)\cdot d(a)\big)=0,
\]
showing that for all $H\subset K\subset G$, $N_H^K(a)$ is again in the kernel. Thus the kernel is also closed under all norm maps, making it a sub-Tambara functor.
\end{proof}
\begin{remark}
Without the ``genuine'' part for a genuine derivation, we could only conclude that the kernel of a derivation was a sub-Green functor.
\end{remark}

We connect now derivations and square zero extensions, showing that the usual results apply with this definition. For this, we need a refinement of Proposition~\ref{prop:NormofSum} describing the norm of a sum, building an increasingly refined series of equations writing norm of a sum as a sum of transfers of norms. 

\begin{definition}
There is a natural grading on $F(G/H,\{0,1\})$ given by 
\[
\deg(f):=\sum_{gH\in G/H} f(gH).
\]
For each $0\leq k\leq [G:H]$, let
\[
T_k:=\{f\in F(G/H,\{0,1\}) | \deg(f)=k\}.
\]
\end{definition}

\begin{proposition}
For each $0\leq k\leq [G:H]$, the subsets $T_k\subset F(G/H,\{0,1\})$ are equivariant subsets, inducing a coproduct decomposition
\[
T_0\amalg\dots\amalg T_{[G:H]}\cong F(G/H,\{0,1\}).
\]
Moreover, the map $G/H\times F(G/H,\{0,1\})\to F(G/H,\{0,1\})$ respects this decomposition in the sense that $G/H\times T_i$ maps to $T_i$.
\end{proposition}
\begin{proof}
Since the degree is defined by summing together all values of $f$ and the $G$-action is given by pre-composition, we have $\deg(f)=\deg(g\cdot f)$ for all $f\in F(G/H,\{0,1\})$ and $g\in G$. In particular, these are equivariant subsets. The decomposition in question then follows from the observation that these are disjoint and that the degree of any function is between $0$ and $[G:H]$. The second part is obvious from the fact that the map in question is just the projection onto the second factor.
\end{proof}

In light of this, we have the following formula which is true for any Tambara functor.

\begin{proposition}\label{prop:NormofSumII}
Let $\m{R}$ be a Tambara functor and let $a, b\in\m{R}(G/H)$. For each $0\leq k\leq [G:H]$, let $f_k\colon T_k\to\ast$ and $g_k\colon G/H\times T_k\to T_k$ be the projections, let $h_k\colon G/H\times T_k\to G/H\amalg G/H$ be the restriction of $\epsilon$ to $T_k$. Then
\[
N_H^G(a+b)=\sum_{k=0}^{[G:H]} T_{f_k} N_{g_k} R_{h_k}(a,b).
\]
\end{proposition}

Proposition~\ref{prop:NormofSumII} allows us to restrict attention to each homogeneous piece. To get our desired result, we need a more explicit formula for $N_{g_k}\circ R_{h_k}$.

%\todo[inline, color=green]{Here is the result that I want to show}
\begin{proposition}\label{prop:CoveringSpaces}
%For each orbit $G\cdot f\subset T_k$, let 
%\[
%T_f:=\epsilon^{-1}(G/H\times\{1\})\cap G/H\times G\cdot f.
%\]
%Then $T_f\to G\cdot f$ is a $k$-to-$1$-map.
Let $T'_k\subset G/H\times T_k$ be $\epsilon^{-1}(G/H\times\{1\})\cap (G/H\times T_k)$. Then $T'_k\to T_k$ is a $k$-fold covering map.
\end{proposition}
\begin{proof}
The $G$-set $T'_k$ is 
\[
T'_k=\{(gH,f) | f(gH)=1\}\subset G/H\times T_k,
\]
so by construction, the fiber over a map $f\in T_k$ has cardinality exactly $k$.
\end{proof}

\begin{theorem}
Let $\m{R}$ be a Tambara functor, let $\mM$ be an $\m{R}$-module, and let $\m{C}\xrightarrow{\epsilon}\m{R}$ be an $\m{S}$-Tambara functor augmented to $\m{R}$. Let $d\colon\m{C}\to\mM$ be a map of Mackey functors. Then
\[
s=\epsilon\semidirect d\colon\m{C}\to\m{R}\semidirect\mM
\]
is a map of $\m{S}$-Tambara functors if and only if $d$ is a genuine $\m{S}$-derivation.
\end{theorem}
\begin{proof}
For notational ease, we suppress explicit mention of $\epsilon$: $\m{R}$ and $\mM$ become $\m{C}$-modules via $\epsilon$ and we use the ordindary notation for such.

Since $d$ is a map of Mackey functors and since Mackey functors form an additive category, $s$ is necessarily a map of Mackey functors. Since the underlying Green functor multiplication is square-zero, the classical argument shows that $s$ is map of Green functors if and only if $d$ is a derivation. We therefore need only show that for all $H\subset K$ and $a\in\m{C}(G/H)$,
\begin{equation}\label{eqn:SufficientNorm}
N_H^K\big(a+d(a)\big)=N_H^K(a)+d\big(N_H^K(a)\big)
\end{equation}
if and only if $d$ is a genuine derivation. By replacing $\m{C}$ with $i_K^\ast\m{C}$, we see that it suffices to verify this for $K=G$.

By Proposition~\ref{prop:NormofSumII}, the left-hand side is
\[
N_H^G\big(a+d(a)\big)=\sum_{k=0}^{[G:H]} T_{f_k} N_{g_k} R_{h_k}\big(a,d(a)\big)=N_H^G(a)+\sum_{k=1}^{[G:H]} T_{f_k}N_{g_k} R_{h_k}\big(a,d(a)\big),
\]
where here $\big(a,d(a)\big)\in\m{R}(G/H)\times\mM(G/H)$. In particular, we conclude that Equation~\ref{eqn:SufficientNorm} holds if and only if
\[
d\big(N_H^G(a)\big)=\sum_{k=1}^{[G:H]} T_{f_k}N_{g_k} R_{h_k}\big(a,d(a)\big).
\]
By Proposition~\ref{prop:CoveringSpaces}, for all $k>1$, on each summand of $T_k$ the map $g_k$ is $k$-to-$1$. In particular, it is a surjective map which is not an isomorphism. The corresponding norm is then necessarily zero on the $\mM$ summand, and hence the product of all of these with terms coming from $\m{R}$ is still zero.  Thus Equation~\ref{eqn:SufficientNorm} holds if and only if
\[
d\big(N_H^G(a)\big)=T_{f_1} N_{g_1} R_{h_1}\big(a,d(a)\big).
\]
The functions $f_1$, $g_1$, and $h_1$ are also easy to understand, since $T_1\cong G/H$, generated by the function which sends $eH$ to $1$ and all other cosets to $0$. The map
\[
h_1\colon G/H\times T_1\to G/H\times\{0,1\}
\]
is then isomorphic to
\[
(G/H\times G/H-\Delta)\amalg G/H\cong G\times_H((i_H^\ast G-H)/H\amalg G/H\to G/H\times\{0,1\}.
\]
This gives us 
\[
R_{h_1}\big(a,d(a)\big)=\big(R_{d_1}(a),d(a)\big).
\]
The map $g_1$ is just the projection onto the second factor $G/H\times G/H\to G/H$. With respect to the decomposition used above, this just becomes
\[
(G/H\times G/H-\Delta)\amalg G/H\to G/H,
\]
where on the first summand, we use the projection onto the second factor and where on the second summand we use the identity. Thus
\[
N_{g_1} R_{h_1}\big(a,d(a)\big)=N_{d_2}R_{d_1}(a)\cdot d(a).
\]
Since $f_1$ is the quotient map $G/H\to\ast$, the associated transfer is just $tr_H^G$. Putting this together shows that Equation~\ref{eqn:SufficientNorm} holds if and only if
\[
d\big(N_H^G(a)\big)=tr_H^G\big(N_{d_2}R_{d_1}(a)\cdot d(a)\big),
\]
which is the definition of $d$ being a genuine derivation.

Since the map $\eta\colon\m{S}\to\m{R}\semidirect\mM$ giving the $\m{S}$-Tambara functor structure factors as the composite 
\[
\m{S}\to\m{R}\xrightarrow{Id\semidirect 0}\m{R}\semidirect\mM,
\] 
we see that $d\circ\eta=0$, automatically.
\end{proof}
%
%\begin{theorem}
%If $s\colon\m{R}\to \m{R}\semidirect\mM$ is a section, then $d=\pi\circ s\colon \m{R}\to \mM$ is a genuine derivation.
%\end{theorem}
%\begin{proof}
%It suffices to show that $\pi\colon\m{R}\semidirect\mM\to\mM$ is a genuine derivation, where $\mM$ is an $\m{R}\semidirect\mM$-module via $\pi$, since the section is a map of Tambara functors by assumption. 
%\end{proof}

\section{K\"ahler Differentials}

%One of the nice consequences of the external formulation of Tambara functors as the $G$-commutative monoids in the category of Mackey functors \cite{MazurArxiv}, \cite{HoyerThesis} is that the category of modules over a Tambara functor has a natural $G$-symmetric monoidal structure.
%
%\begin{proposition}[{\cite[]{HHLocalization}}]
%If $\m{R}$ is a Tambara functor, then the category $\Mods{\m{R}}$ has internal norm maps given by
%\[
%{}_{\m{R}} N_H^G(\mM):= \m{R}\Boxover{N_H^Gi_H^\ast\m{R}} N_H^G\mM,
%\]
%where the map $N_H^Gi_H^\ast \m{R}\to\m{R}$ is simply the counit of the norm-forget adjunction.
%\end{proposition}
%
%In particular, it makes sense to talk about Tambara functors and non-unital Tambara functors in the category of $\m{R}$-modules, as these are just the $G$-commutative monoids and non-unital $G$-commutative monoids in this category. 

One of the tricky parts of generalizing the notion of K\"ahler differentials is finding the right way to work with ideals in the context of Tambara functors. Work of Nakaoka describes the right version of Tambara ideals, and we build on that here \cite{NakaokaIdeals}. In the language of Definition~\ref{def:NonUnitalRAlgebra}, if $\m{R}$ is a Tambara functor, then a Tambara ideal is simply a sub-non-unital $\m{R}$-Tambara functor.

\begin{definition}
Let $\m{R}$ be a Tambara functor and let $\m{I}$ be a non-unital $\m{R}$-Tambara functor. Let
\[
\m{I}^{>1}:=\sum_{H\subset G, T\in\Set^H, |T|>1} \Ind_H^G N^T i_H^\ast \m{I}\subset \m{I},
\]
where here $\Ind_H^GN^T i_H^\ast\m{I}$ stands for the image of the corresponding structure map.

We call this the \defemph{submodule of genuine equivariant decomposable elements}.
\end{definition}
%For defining the K\"ahler differentials in a Tambara functor, it is helpful to work more generally, defining various forms of K\"ahler differentials for $\cO$-Tambara functors \cite{BHIncomplete}.
%
%\begin{definition}
%If $\m{I}$ is a non-unital $\cO$-Tambara functor, then let 
%\[
%\m{I}^{>1}_{\mathcal O}=\sum_{H\subset G, T\in\cO(G/H), |T|>1} \Ind_H^G N^T i_H^\ast \m{I}\subset \m{I}
%\]
%be the sub-module of $\cO$-decomposable elements.
%\end{definition}
%
\begin{proposition}
For any non-unital Tambara functor $\m{I}$ in $\m{R}$-modules, $\m{I}^{>1}$ is a Tambara ideal of $\m{I}$.
%
%Let $\cO$ have the additional feature that if $H/K$ is an admissible $H$-set for $\cO$ then $J/K$ is an admissible $J$-set for all $H\subset J\subset G$. Then $\m{I}^{>1}_{\cO}$ is a Tambara ideal of $\m{I}$.
\end{proposition}
\begin{proof}
Interpreting the norm as a generalized product over a possibly non-trivial $G$-set, we see that $\m{I}^{> 1}$ is the sub-Mackey functor generated by possible products with more than one factor. This is visibly closed under products by elements in $\m{I}$ and by products in itself. The universal formulae for norms of sums and of transfers also preserves the underlying cardinality of the exponents, showing that linear combinations are also still in this collection.
%
%The latter condition is a simple restatement of the fact that for any $H\subset J$, if $N_K^H(a)\in\m{I}^{>1}_{\cO}$, then 
%\[
%N_H^J\big(N_K^H(a)\big)=N_K^J(a)
%\]
%is again in $\m{I}^{>1}_{\cO}$, showing that this is closed under arbitrary norms.
\end{proof}

\begin{proposition}
If $f\colon\m{B}\to\m{R}$ is a map of Tambara functors, then the kernel of $f$ is a non-unital Tambara functor.
\end{proposition}
\begin{proof}
The zero-map is a map of non-unital Tambara functors. Since the kernel is the equalizer of $f$ and the zero map and since the category of non-unital Tambara functors is complete, the kernel is a non-unital Tambara functor. 
\end{proof}

\begin{definition}\label{def:GenuineKahler}
Let $\m{S}$ be a Tambara functor and let $\m{R}$ be a Tambara functor under $\m{S}$. Let $\m{I}$ denote the kernel of the multiplication map
\[
\m{R}\Boxover{\m{S}} \m{R}\to\m{R}.
\]

The $\m{R}$-module
\[
\OmegaOneG:= \m{I}/\m{I}^{>1},
\]
is defined to be the module of {\defemph{genuine K\"{a}hler differentials}}, and let 
\[
d\colon \m{R}\to \OmegaOneG
\]
be the difference between the left and right inclusions $\m{R}\to\m{R}\Boxover{\m{S}}\m{R}$.
\end{definition}

%As $\cO$ varies, the assignment 
%\[
%\cO\mapsto\OmegaOneO
%\]
%gives a functor from the poset of indexing systems to the category of $\m{R}$-modules and surjections. In particular, there are two special cases we single out.
%
%\begin{definition}
%If $\cO=\cO^{tr}$, then we will call $\OmegaOneO$ the module of {\defemph{Green K\"{a}hler differentials}}.
%
%If $\cO=\cO^{all}$, then we will call $\OmegaOneO$ the modules of {\defemph{Genuine K\"{a}hler differentials}}. We will also denote this  $\OmegaOneG$.
%\end{definition}
%
%The Green K\"ahler differentials are initial amongst all $\OmegaOneO$, while the genuine ones are terminal. The following is immediate from the construction.
%
%\begin{proposition}
%The module of Green K\"{a}hler differentials for $\m{R}$ over $\m{S}$ is the ordinary module of K\"{a}hler differentials for the underlying Green functors.
%\end{proposition}

\begin{proposition}
The $\m{R}$-module $\OmegaOneG$ is generated by the image of $d$.
\end{proposition}
\begin{proof}
It suffices to prove the simpler, Green functor version of this statement, where we let $\m{I}^2$ simply be the usual box-square of $\m{I}$ and show that $\m{I}/\m{I}^2$ is generated by the corresponding image of $d$. Since $\m{I}^2\subset\m{I}^{>1}$, this implies our result.

Here, we copy the classical argument. The collection $\m{R}(G/H)\otimes_{\m{S}(G/H)} \m{R}(G/H)$ for all $H\subset G$ generates $\m{R}\Box_{\m{S}}\m{R}$ as a Mackey functor. The map $\m{R}\Boxover{\m{S}}\m{R}\to\m{R}$ is a map of Tambara functors, and 
\[
\m{R}\xrightarrow{\eta_L-\eta_R}\m{R}\Boxover{\m{S}}\m{R}
\]
is a map of Mackey functors. Since the ordinary tensor products generate as Mackey functors, we can simply copy the classical proof, giving the result.
\end{proof}

\begin{lemma}\label{lem:UniversalDerivation}
The map $d\colon \m{R}\to\OmegaOneG$ is a genuine $\m{S}$-derivation.
\end{lemma}
\begin{proof}
The sequence of $\m{R}$-modules
\[
0\to\m{I}\to \m{R}\Boxover{\m{S}}\m{R}\to\m{R}\to 0
\]
is split by the left unit. This splitting gives an identification
\[
(\m{R}\Boxover{\m{S}}\m{R})/\m{I}^{>1}\cong\m{R}\semidirect\OmegaOneG
\]
of Tambara functors augmented over $\m{R}$. The map $d$ is the difference between the left and right units, and since both the left and right units are maps of Tambara functors, $d$ is a genuine derivation. Since the box product is over $\m{S}$, both the left and right units agree on $\m{S}$, and hence $d$ is a genuine $\m{S}$-derivation.
\end{proof}

\begin{theorem}
If $\mM$ is an $\m{R}$-module, then there is a natural isomorphism
\[
\Der^{1,G}_{\m{S}}(\m{R},\mM)\cong \Hom_{\m{R}}(\OmegaOneG,\mM).
\]
\end{theorem}
\begin{proof}
By Corollary~\ref{lem:UniversalDerivation}, the map $d\colon\m{R}\to\OmegaOneG$ is a genuine derivation. Proposition~\ref{prop:Naturality} shows then that given any map of $\m{R}$-modules $\OmegaOneG\to\mM$, we can compose with $d$ to get a derivation into $\mM$.

For the other direction, let $d$ be a genuine derivation $\m{R}\to\mM$. Then $d$ induces a map of Tambara functors
\[
\m{R}\xrightarrow{Id\semidirect d} \m{R}\semidirect\mM
\]
augmented over $\m{R}$. By extending scalars over $\m{S}$ back to $\m{R}$, we get a map
\[
\m{R}\Boxover{\m{S}}\m{R}\to\m{R}\semidirect\mM
\]
of Tambara functors augmented over $\m{R}$, where the source is augmented by the multiplication map. In particular, the augmentation ideal $\m{I}$ maps to $\mM$. Since $\mM$ is equivariantly square zero, this map descends to a map 
\[
(\m{R}\Boxover{\m{S}}\m{R})/\m{I}^{>1}\cong \m{R}\semidirect\OmegaOneG\to \m{R}\semidirect\mM
\]
of Tambara functors augmented over $\m{R}$. This gives us a map of $\m{R}$-modules
\[
\Omega^{1,G}_{\m{R}/\m{S}}\to\mM.
\]
Since $\OmegaOneG$ is generated by the image of $d\colon \m{R}\to\OmegaOneG$, we know that this map is unique.
\end{proof}

\begin{corollary}
For any Tambara functor $\m{R}$ and any $\m{R}$-module $\mM$, the set $\Der_{\m{S}}^{1,G}(\m{R},\mM)$ has a natural extension to a Mackey functor whose value at $G/H$ is
\[
\Der_{i_H^\ast\m{S}}^{1,H}(i_H^\ast\m{R},i_H^\ast\mM).
\]
\end{corollary}

\begin{definition}
A map $\m{S}\to\m{R}$ of Tambara functors is {\defemph{formally \'etale}} if $\OmegaOneG=0$ and $\m{R}$ is a flat $\m{S}$-module.
\end{definition}

Just as classically, localizations are formally \'etale. Here, we can invert a set of elements that come from the value of the Tambara functor at various $G$-sets $T$ \cite{BHIncomplete}. We first show that localizations in Tambara functors are flat. 

\begin{proposition}\label{prop:LocalizationisFlat}
Let $S$ be a collection of elements from $\m{R}$. Then $\m{R}[S^{-1}]$ is a flat $\m{R}$-module.
\end{proposition}
\begin{proof}
If all elements of $S$ come from $\m{R}(G/G)$, then the localization $\m{R}[S^{-1}]$ can be formed as a filtered colimit of copies of $\m{R}$ along maps of the form $N_e^G(\res_e^G(s))$, where $s\in S$. In particular, this is flat. 

More generally, since we are forming the localization in Tambara functors, inverting any $s\in\m{R}(G/H)$ also inverts $N_H^G(s)$, and by the multiplicative double coset formula, inverting $N_H^G(s)$ also inverts $s$. In particular, it suffices to consider only localizations at a set of elements in $\m{R}(G/G)$ and the result follows.
\end{proof}

\begin{remark}
It was essential here that we could write any localization as a filtered colimit of free modules which in turn required that we could write any localization as one which inverts a collection of elements in $\m{R}(G/G)$. For any arbitrary Green or incomplete Tambara functor, this is no longer the case, so it is not obvious that localization is a flat operation here.
\end{remark}

\begin{remark}
One of the surprising consequences of the proof of Proposition~\ref{prop:LocalizationisFlat} is that the basic Zariski open sets in Nakaoka's spectrum of a Tambara functor arise by inverting elements in $\m{R}(G/G)$, rather than in any other level of the Tambara functor \cite{NakaokaCyclic}. This suggests a much more rigid behavior than initially expected.
\end{remark}

\begin{proposition}
If $S$ is a multiplicative subset in $\m{R}$, then $\m{R}\to \m{R}[S^{-1}]$ is formally \'etale.
\end{proposition}
\begin{proof}
Both $\m{R}[S^{-1}]$ and its box-square over $\m{R}$ satisfy the same universal property, so we conclude that the multiplication map
\[
\m{R}[S^{-1}]\Boxover{\m{R}} \m{R}[S^{-1}]\to\m{R}[S^{-1}]
\]
is an isomorphism. In particular, $\m{I}$ defined above is itself zero.
\end{proof}

\bibliographystyle{plain}

\bibliography{Omega1G}

\end{document}

%% file: MAHMacros.tex
\newcommand{\support}[1]{The author was supported by {#1}.}

\newcommand{\NSFThree}{NSF Grant DMS-1509652}

\newcommand{\Boxover}[1]{\underset{#1}{\Box}}

\newcommand{\Ind}{\big\uparrow} %Matching Mackey notation
\DeclareMathOperator{\Hom}{Hom}
\DeclareMathOperator{\Der}{Der}

\newcommand{\res}{res}
\newcommand{\tr}{tr}
\newcommand{\Tr}{tr}

\newcommand{\m}[1]{{\protect\underline{#1}}}

\newcommand{\mM}{\m{M}}

\newcommand{\mTr}{\m{\Tr}}

\newcommand{\cc}[1]{\mathcal #1}
\newcommand{\cC}{\cc{C}}

\newcommand{\cP}{\cc{P}}
\newcommand{\TAQ}{\textnormal{TAQ}}
\newcommand{\CoInd}{\textnormal{CoInd}}

\newcommand{\Set}{\mathcal Set}

\newcommand{\Ab}{\mathcal Ab}

\newcommand{\Tamb}{\mathcal Tamb}

\newcommand{\Mackey}{\mathcal Mackey}

\newcommand{\Mod}{\mathcal Mod}
\newcommand{\Mods}[1]{#1\mhyphen\Mod}

\newcommand{\STamb}{\m{S}\mhyphen\Tamb}
\newcommand{\SAb}{\m{S}\mhyphen\Ab}

\newcommand{\OmegaOneG}{\Omega^{1,G}_{\m{R}/\m{S}}}
\newcommand{\semidirect}{\ltimes}

\newcommand{\mC}{\m{\cC}}

\mathchardef\mhyphen=45

\numberwithin{equation}{section}

\newtheorem{theorem}{Theorem}[section]
\newtheorem{lemma}[theorem]{Lemma}
\newtheorem{corollary}[theorem]{Corollary}

\newtheorem{proposition}[theorem]{Proposition}

\newtheorem*{theorem*}{Theorem}

\theoremstyle{remark}
\newtheorem{remark}[theorem]{Remark}

\theoremstyle{definition}

\newtheorem{definition}[theorem]{Definition}

\newcommand{\defemph}[1]{\textbf{#1}}

%% file: Omega1G.bbl
\begin{thebibliography}{10}

\bibitem{Andre67}
Michel Andr\'e.
\newblock {\em M\'ethode simpliciale en alg\`ebre homologique et alg\`ebre
  commutative}.
\newblock Lecture Notes in Mathematics, Vol. 32. Springer-Verlag, Berlin-New
  York, 1967.

\bibitem{Basterra99}
M.~Basterra.
\newblock Andr\'e-{Q}uillen cohomology of commutative {$S$}-algebras.
\newblock {\em J. Pure Appl. Algebra}, 144(2):111--143, 1999.

\bibitem{BasterraMandell}
Maria Basterra and Michael~A. Mandell.
\newblock Homology and cohomology of {$E_\infty$} ring spectra.
\newblock {\em Math. Z.}, 249(4):903--944, 2005.

\bibitem{BHIncomplete}
Andrew~J. Blumberg and Michael~A. Hill.
\newblock Incomplete tambara functors.
\newblock arxiv.org, 2016.

\bibitem{HHLocalization}
Michael~A. Hill and Michael~J. Hopkins.
\newblock Equivariant symmetric monoidal structures.
\newblock arxiv.org: 1610.03114, 2012.

\bibitem{KrizTAQ}
Igor Kriz.
\newblock Towers of {$E_\infty$}-ring spectra with an application to {$BP$}.
\newblock Preprint.

\bibitem{MazurArxiv}
Kristen Mazur.
\newblock An equivariant tensor product on {M}ackey functors.
\newblock arxiv.org: 1508.04062, 2015.

\bibitem{NakaokaIdeals}
Hiroyuki Nakaoka.
\newblock Ideals of {T}ambara functors.
\newblock {\em Adv. Math.}, 230(4-6):2295--2331, 2012.

\bibitem{NakaokaCyclic}
Hiroyuki Nakaoka.
\newblock The spectrum of the {B}urnside {T}ambara functor on a finite cyclic
  {$p$}-group.
\newblock {\em J. Algebra}, 398:21--54, 2014.

\bibitem{Quillen70}
Daniel Quillen.
\newblock On the (co-) homology of commutative rings.
\newblock In {\em Applications of {C}ategorical {A}lgebra ({P}roc. {S}ympos.
  {P}ure {M}ath., {V}ol. {XVII}, {N}ew {Y}ork, 1968)}, pages 65--87. Amer.
  Math. Soc., Providence, R.I., 1970.

\bibitem{Strickland}
Neil Strickland.
\newblock Tambara functors.
\newblock arxiv.org: 1205.2516, 2012.

\bibitem{Tambara}
D.~Tambara.
\newblock On multiplicative transfer.
\newblock {\em Comm. Algebra}, 21(4):1393--1420, 1993.

\end{thebibliography}
